\documentclass[12pt,reqno]{amsart}
\usepackage{eurosym}
\usepackage{amsfonts}
\usepackage{amsmath}
\usepackage{amssymb,latexsym}
\usepackage{enumerate}
\usepackage[bookmarksnumbered, colorlinks, plainpages]{hyperref}
\usepackage{amsbsy}
\usepackage{amscd}
\usepackage{epsfig}
\usepackage{graphicx}
\usepackage{epstopdf}

\newtheorem{theorem}{Theorem}[section]

\newtheorem{corollary}{Corollary}[section]
\newtheorem{definition}{Definition}[section]

\newtheorem{example}{Example}[section]
\newtheorem{remark}{Remark}[section]
\numberwithin{equation}{section}

\begin{document}
\date{{\scriptsize Received: , Accepted: .}}
\title[Simulation Functions ]{Geometric Properties of Fixed Points and
Simulation Functions }
\subjclass[2010]{Primary 54H25; Secondary 47H09, 47H10.}
\keywords{Fixed circle, fixed disc, fixed ellipse, simulation function.}
\author[N.\ \"{O}ZG\"{U}R]{Nihal \"{O}ZG\"{U}R$^*$}
\address{Bal\i kesir University, Department of Mathematics, 10145 Bal\i
kesir, Turkey}
\email{nihal@balikesir.edu.tr}
\author[N. TA\c{S}]{Nihal TA\c{S}}
\address{Bal\i kesir University, Department of Mathematics, 10145 Bal\i
kesir, Turkey}
\email{nihaltas@balikesir.edu.tr}
\thanks{$^{\ast }$Corresponding author: Bal\i kesir University, Department
of Mathematics, 10145 Bal\i kesir, TURKEY, nihal@balikesir.edu.tr}
\maketitle

\begin{abstract}
Geometric properties of the fixed point set $Fix(f)$ of a self-mapping $f$
on a metric or a generalized metric space is an attractive issue. The set $%
Fix(f)$ can contain a geometric figure (a circle, an ellipse, etc.) or it
can be a geometric figure. In this paper, we consider the set of simulation
functions for geometric applications in the fixed point theory both on
metric and some generalized metric spaces ($S$-metric spaces and $b$-metric
spaces). The main motivation of this paper is to investigate the geometric
properties of non unique fixed points of self-mappings via simulation
functions.
\end{abstract}

\bigskip




%

%

\section{\textbf{Introduction}}

\label{intro} Recently, the set of simulation functions defined in \cite%
{Khojasteh} has been used for the solutions of many recent problems such as
fixed-circle problem (resp. fixed-disc problem) and Rhoades' open problem on
discontinuity (see \cite{Mlaiki 2019, Ozgur-simulation, Pant 2020}). On the
other hand, simulation functions have been studied by various aspects in
metric fixed-point theory (see for example \cite{Chanda 2019, Karapinar
2017, Khojasteh, Kostic 2019, Olgun 2016, Roldan, Roldan 2015}). For
example, in \cite{Olgun 2016}, Olgun et al. gave a new class of Picard
operators on complete metric spaces via simulation functions. Simulation
functions have been used to study the best proximity points in metric
spaces. For example, Kosti\'{c} et al. presented several best proximity
point results involving simulation functions (for more details see \cite%
{Karapinar 2017, Kostic 2019} and the references therein).

In \cite{Roldan}, the set of simulation functions has been enlarged by A. F.
Rold\'{a}n-L\'{o}pez-de-Hierro et al. Every simulation function in the
original Khojasteh et al.'s sense is also a simulation function in A. F. Rold%
\'{a}n-L\'{o}pez-de-Hierro et al.'s sense but the converse is not true (see
\cite{Roldan} for more details). In this paper, we focus on the set of
simulation functions in both sense and using their properties, we consider
some recent problems in fixed-point theory.

Recall that the function $\zeta :[0,\infty )\times \lbrack 0,\infty
)\rightarrow
\mathbb{R}
$ is said to be a simulation function in the Khojasteh et al.'s sense, if
the following hold:

$(\zeta _{1})$ $\zeta (0,0)=0,$

$(\zeta _{2})$ $\zeta (t,s)<s-t$ for all $s,t>0$,

$(\zeta _{3})$ If $\{t_{n}\}$, $\{s_{n}\}$ are sequences in $(0,\infty )$
such that%
\begin{equation*}
\underset{n\rightarrow \infty }{\lim }t_{n}=\underset{n\rightarrow \infty }{%
\lim }s_{n}>0\text{,}
\end{equation*}%
then%
\begin{equation*}
\underset{n\rightarrow \infty }{\lim \sup }\zeta (t_{n},s_{n})<0\text{.}
\end{equation*}

The set of all the simulation functions is denoted by $\mathcal{Z}$ (see
\cite{Chanda 2019} and \cite{Khojasteh} for more details).

In \cite{Roldan}, Rold\'{a}n-L\'{o}pez-de-Hierro et al. modified this
definition of simulation functions and so enlarged the family of all
simulation functions. To do this, only the condition $(\zeta _{3})$ was
replaced by the following condition $(\zeta _{3})^{\ast }$ as follows:

$(\zeta _{3})^{\ast }$ If $\{t_{n}\}$, $\{s_{n}\}$ are sequences in $%
(0,\infty )$ such that%
\begin{equation*}
\underset{n\rightarrow \infty }{\lim }t_{n}=\underset{n\rightarrow \infty }{%
\lim }s_{n}>0\text{ }
\end{equation*}%
and
\begin{equation*}
t_{n}<s_{n}\text{ for all }n\in
\mathbb{N}
\text{,}
\end{equation*}%
then%
\begin{equation*}
\underset{n\rightarrow \infty }{\lim \sup }\zeta (t_{n},s_{n})<0\text{.}
\end{equation*}

Every simulation function in the Khojasteh et al.'s sense is also a
simulation function in the Rold\'{a}n-L\'{o}pez-de-Hierro et al.'s sense,
the converse is not true (for example see Example 3.3 in \cite{Roldan}). In
applications of the simulation functions to the study of discontinuity
problem and the geometric study of the non unique fixed points of
self-mappings, the condition $(\zeta _{3})$ is not used. So, both
definitions of the simulation functions and examples can be used to study
such kind applications. Some examples of simulation functions $\zeta
:[0,\infty )\times \lbrack 0,\infty )\rightarrow
\mathbb{R}
$ are

$1)$ $\zeta (t,s)=\lambda s-t$, where $\lambda \in \lbrack 0,1)$,

$2)$ $\zeta (t,s)=s-\varphi (s)-t$, where $\varphi :[0,\infty )\rightarrow
\lbrack 0,\infty )$ is a continuous function such that $\varphi (t)=0$ if
and only if $t=0$,

$3)$ $\zeta (t,s)=s\phi (s)-t$, where $\phi :[0,\infty )\rightarrow \lbrack
0,1)$ is a mapping such that $\underset{t\rightarrow r^{+}}{\lim \sup \phi
(t)}<1$ for all $r>0$,

$4)$ $\zeta (t,s)=\eta (s)-t$, where $\varphi :[0,\infty )\rightarrow
\lbrack 0,\infty )$ be an upper semi-continuous mapping such that $\eta
(t)<t $ for all $t>0$ and $\eta (0)=0$,

$5)$ $\zeta (t,s)=s-\int\limits_{0}^{t}\psi (t)dt$, where $\psi :[0,\infty
)\rightarrow \lbrack 0,1)$ is a function such that $\int\limits_{0}^{t}\psi
(t)dt$ exists and $\int\limits_{0}^{\varepsilon}\psi (u)du>\varepsilon $ for
each $\varepsilon >0$.

The main motivation of this paper is the study of the geometric properties
of the fixed point set of a self-mapping in the non unique fixed point case.
There are some examples of self-mappings where the fixed point set of the
self-mapping contains a geometric figure such as a circle, a disc or an
ellipse. For example, let us consider the metric space $\left(
\mathbb{C}
,d\right) $ with the metric defined for the complex numbers $%
z_{1}=x_{1}+iy_{1}$ and $z_{2}=x_{2}+iy_{2}$ as follows:
\begin{equation*}
d\left( z_{1},z_{2}\right) =\sqrt{\frac{\left( x_{1}-x_{2}\right) ^{2}}{9}%
+4\left( y_{1}-y_{2}\right) ^{2}}\text{,}
\end{equation*}%
where $d$ is the metric induced by the norm function $\left\Vert
z\right\Vert =\left\Vert x+iy\right\Vert =$ $\sqrt{\frac{x^{2}}{9}+4y^{2}}$.
Consider the circle $C_{0,1}$ and define the self-mapping $g$ on $\mathbb{C}$
by
\begin{equation*}
gz=\left\{
\begin{array}{ccc}
z & ; & x\geq 0,y\geq 0\text{ or }x\leq 0,y\leq 0 \\
\frac{36z}{-35(z^{2}+\overline{z}^{2})+74z\overline{z}} & ; & x<0,y>0\text{
or }x>0,y<0%
\end{array}%
\right. \text{,}
\end{equation*}%
for each $z=x+iy\in \mathbb{C}$, where $\overline{z}=x-iy$ is the complex
conjugate of $z$. Then, it is easy to verify that the fixed point set of $g$
contains the circle $C_{0,1}$, that is, $C_{0,1}$ is a fixed circle of $g$.
We use the notion of \textquotedblleft inversion in an
ellipse\textquotedblright\ to construct this self-mapping (see Proposition 1
given in \cite{Ramirez}).

There are several papers for the cases a fixed circle and a fixed disc (see
\cite{Aydi 2019, Mlaiki 2018, Mlaiki arxiv, Ozgur-malaysian,
Ozgur-simulation, Pant Ozgur Tas 2019, Pant Ozgur Tas accepted,
Tas-Ozgur-Mlaiki-2, Tas 2020} and the references therein). The fixed ellipse
case is also considered in the recent studies \cite{Ercinar} and \cite{Joshi}%
. In \cite{Ercinar}, the cases a fixed Apollonius circle, fixed hyperbola
and fixed Cassini curve are considered extensively on metric and some
generalized metric spaces. Therefore, the study of geometric properties of
the fixed point set of a self-mapping seems to be an interesting problem in
case where the fixed point is non unique. In this paper, we study on the
geometric properties of the fixed point set of a self-mapping via simulation
functions on metric (resp. $S$-metric and $b$-metric) spaces. The
relationships among a metric, an $S$-metric and a $b$-metric are well known,
so we refer the reader to \cite{Ozgur 2017}, \cite{Sedghi 2012} and \cite%
{Sedghi 2014} for more details.

\section{\textbf{Simulation functions and the geometry of fixed points}}

\label{sec:1}

In this section, we study on geometric properties of the fixed point set $%
Fix(f)=\left\{ x\in X:fx=x\right\} $ for a self-mapping $f$ of a metric
(resp. $S$-metric, $b$-metric) space. The set of simulation functions has
been used in \cite{Mlaiki 2019}, \cite{Ozgur-simulation} and \cite{Pant 2020}
to obtain new results on the fixed-circle (resp. fixed-disc) problem. In
\cite{Pant 2020}, together with some properties of simulation functions, the
numbers $M(x,y)$ and $\rho $ defined by%
\begin{equation}
M(x,y)=\max \left\{
\begin{array}{c}
ad(x,fx)+(1-a)d(y,fy), \\
(1-a)d(x,fx)+ad(y,fy),\frac{d(x,fy)+d(y,fx)}{2}%
\end{array}%
\right\} ,\text{ }0\leq a<1\text{,}  \label{definition of M}
\end{equation}%
\begin{equation}
\rho =\inf \left\{ d(x,fx):x\neq fx,x\in X\right\} \text{,}  \label{number}
\end{equation}%
and an auxiliary function $\varphi :%
\mathbb{R}
^{+}\rightarrow
\mathbb{R}
^{+}$ satisfying $\varphi (t)<t$ for each $t>0$, were used to get new
fixed-circle (resp. fixed-disc) results. Using these numbers and an
auxiliary function, we present new results on the geometric study of the
fixed point set of a self-mapping.

\subsection{\textbf{Geometric study of fixed points on metric spaces}}

Let $(X,d)$ be a metric space and $f:X\rightarrow X$ be a self-mapping.
First, we recall that the circle $C_{x_{0},\rho }=\left\{ x\in X:d\left(
x,x_{0}\right) =\rho \right\} $ (resp. the disc $D_{x_{0},\rho }=\left\{
x\in X:d\left( x,x_{0}\right) \leq \rho \right\} $) is a fixed circle (resp.
a fixed disc) of $f$ if $fx=x$ for all $x\in C_{x_{0},\rho }$ (resp. for all
$x\in D_{x_{0},\rho }$) (see \cite{Ozgur-malaysian}, \cite{Ozgur-simulation}%
). More generally, a geometric figure $\mathcal{F}$ (a circle, an ellipse, a
hyperbola, a Cassini curve etc.) contained in the fixed point set $Fix\left(
f\right) $ is called a\textit{\ fixed figure} (a fixed circle, a fixed
ellipse, a fixed hyperbola, a fixed Cassini curve, etc.) of the self-mapping
$f$.

Let $E_{r}(x_{1},x_{2})$ be the ellipse defined as
\begin{equation*}
E_{r}(x_{1},x_{2})=\left\{ x\in X:d\left( x,x_{1}\right) +d\left(
x,x_{2}\right) =r\right\} \text{.}
\end{equation*}%
Clearly, we have
\begin{equation*}
r=0\Rightarrow x_{1}=x_{2}\text{ and }E_{r}(x_{1},x_{2})=C_{x_{1},r}=\left\{
x_{1}\right\} .
\end{equation*}

Now, we use the set of simulation functions and the number $\rho $ to obtain
some results for the case where the fixed point set $Fix(f)$ contains an
ellipse or an ellipse with its interior.

\begin{theorem}
\label{thm21} Let $(X,d)$ be a metric space, $f:X\rightarrow X$ be a
self-mapping, $\zeta \in \mathcal{Z}$ be a simulation function and the
number $\rho $ be defined as in $($\ref{number}$)$. If there exist some
points $x_{1},x_{2}\in X$ such that

$(a)$ For all $x\in E_{\rho }(x_{1},x_{2})$, there exists $\delta (\rho )>0$
satisfying%
\begin{equation*}
\frac{\rho }{2}\leq M(x,x_{1})+M(x,x_{2})<\frac{\rho }{2}+\delta (\rho
)\Longrightarrow d(fx,x_{1})+d(fx,x_{2})\leq \rho \text{,}
\end{equation*}

$(b)$ For all $x\in X$,%
\begin{equation*}
d(fx,x)>0\Longrightarrow \zeta \left( d(fx,x),M(x,x_{1})\right) \geq 0\text{
and }\zeta \left( d(fx,x),M(x,x_{2})\right) \geq 0\text{,}
\end{equation*}

$(c)$ For all $x\in X$,%
\begin{equation*}
d(fx,x)>0\Longrightarrow \zeta \left( d(fx,x),\frac{%
d(x,x_{1})+d(fx,x_{1})+d(x,x_{2})+d(fx,x_{2})}{2}\right) \geq 0\text{,}
\end{equation*}%
then $fx_{1}=x_{1}$, $fx_{2}=x_{2}$ and $Fix(f)$ contains the ellipse $%
E_{\rho }(x_{1},x_{2})$.
\end{theorem}

\begin{proof}
We have
\begin{equation*}
M(x_{1},x_{1})=d(fx_{1},x_{1})\text{ and }M(x_{2},x_{2})=d(fx_{2},x_{2})%
\text{.}
\end{equation*}%
First, we show that $fx_{1}=x_{1}$ and $fx_{2}=x_{2}$. If $fx_{1}\neq x_{1}$
and $fx_{2}\neq x_{2}$ then $d(fx_{1},x_{1})>0$ and $d(fx_{2},x_{2})>0$.
Using the condition $(\zeta _{2})$, we find
\begin{eqnarray*}
\zeta \left( d(fx_{1},x_{1}),M(x_{1},x_{1})\right) &=&\zeta \left(
d(fx_{1},x_{1}),d(fx_{1},x_{1})\right) \\
&<&d(fx_{1},x_{1})-d(fx_{1},x_{1})=0
\end{eqnarray*}%
and%
\begin{eqnarray*}
\zeta \left( d(fx_{2},x_{2}),M(x_{2},x_{2})\right) &=&\zeta \left(
d(fx_{2},x_{2}),d(fx_{2},x_{2})\right) \\
&<&d(fx_{2},x_{2})-d(fx_{2},x_{2})=0\text{,}
\end{eqnarray*}%
which are contradictions with the condition $(b)$. Hence it should be $%
fx_{1}=x_{1}$ and $fx_{2}=x_{2}$.

If $\rho =0$, then we have $E_{\rho }(x_{1},x_{2})=C_{x_{1},\rho }=\{x_{1}\}$
\ and $x_{1}=x_{2}$. Hence, the proof is completed.

Now we assume $\rho \neq 0$. Let $x\in E_{\rho }(x_{1},x_{2})$ be any point
such that $fx\neq x$. Then $d(x,fx)>0$ and we have%
\begin{equation*}
M(x,x_{1})=\max \left\{ ad(x,fx),(1-a)d(x,fx),\frac{d(x,x_{1})+d(x_{1},fx)}{2%
}\right\}
\end{equation*}%
and
\begin{equation*}
M(x,x_{2})=\max \left\{ ad(x,fx),(1-a)d(x,fx),\frac{d(x,x_{2})+d(x_{2},fx)}{2%
}\right\} \text{.}
\end{equation*}%
Using the condition $(a)$, we get%
\begin{equation}
\frac{\rho }{2}\leq M(x,x_{1})+M(x,x_{2})<\frac{\rho }{2}+\delta (\rho
)\Longrightarrow d(fx,x_{1})+d(fx,x_{2})\leq \rho \text{.}  \label{eqn21}
\end{equation}%
Now, using the inequality (\ref{eqn21}) and the conditions $(c)$, $(\zeta
_{2})$, we obtain%
\begin{eqnarray*}
0 &\leq &\zeta \left( d(fx,x),\frac{%
d(x,x_{1})+d(fx,x_{1})+d(fx,x_{2})+d(x,x_{2})}{2}\right) \\
&<&\frac{d(x,x_{1})+d(fx,x_{1})+d(fx,x_{2})+d(x,x_{2})}{2}-d(fx,x) \\
&=&\frac{d(x,x_{1})+d(x,x_{2})}{2}+\frac{d(fx,x_{1})+d(fx,x_{2})}{2}-d(fx,x)
\\
&\leq &\frac{\rho }{2}+\frac{\rho }{2}-d(fx,x)=\rho -d(fx,x)
\end{eqnarray*}%
and hence%
\begin{equation*}
d(fx,x)<\rho \text{.}
\end{equation*}%
This is a contradiction by the definition of the number $\rho $. Because of
this contradiction, it should be $fx=x$. Consequently, we have $E_{\rho
}(x_{1},x_{2})\subset Fix(f)$.
\end{proof}

\begin{remark}
\label{remark21} If $x_{1}=x_{2}$ then we have $E_{\rho
}(x_{1},x_{2})=C_{x_{1},\frac{\rho }{2}}$ and Theorem \ref{thm21} is reduced
to a fixed-circle theorem as follows$:$
\end{remark}

\begin{theorem}
\label{thm22} Let $(X,d)$ be a metric space, $f:X\rightarrow X$ be a
self-mapping, $\zeta \in \mathcal{Z}$ be a simulation function and the
number $\rho $ be defined as in $($\ref{number}$)$. If there exists some
point $x_{0}\in X$ such that

$(a)$ For all $x\in C_{x_{0},\rho }$, there exists $\delta (\rho )>0$
satisfying%
\begin{equation*}
\frac{\rho }{4}\leq M(x,x_{0})<\frac{\rho }{4}+\delta (\rho )\Longrightarrow
d(fx,x_{0})\leq \rho \text{,}
\end{equation*}

$(b)$ For all $x\in X$,%
\begin{equation*}
d(fx,x)>0\Longrightarrow \zeta \left( d(fx,x),M(x,x_{0})\right) \geq 0\text{,%
}
\end{equation*}%
then $fx_{0}=x_{0}$ and the set $Fix(f)$ contains the circle $C_{x_{0},\rho
} $.
\end{theorem}

\begin{proof}
The proof follows by Theorem \ref{thm21} and Remark \ref{remark21}.
\end{proof}

\begin{example}
\label{exm21} Let $X=\left\{ -3,-1,1,3,12,18\right\} $ with the metric $%
d(x,y)=\left\vert x-y\right\vert $. Define the self-mapping $f:X\rightarrow
X $ by
\begin{equation*}
fx=\left\{
\begin{array}{ccc}
x+6 & , & x=12 \\
x & , & x\in \left\{ -3,-1,1,3,18\right\}%
\end{array}%
\right. .
\end{equation*}%
Then the self-mapping $f$ satisfies the conditions of Theorem \ref{thm21}
for the points $x_{1}=-1$ and $x_{2}=1$ and the simulation function $\zeta
(t,s)=\frac{1}{2}s-t$. Indeed, we have
\begin{equation*}
\rho =\min \left\{ d(x,fx):x\in X,x\neq fx\right\} =6
\end{equation*}%
and
\begin{equation*}
E_{6}(-1,1)=\left\{ -3,3\right\} .
\end{equation*}%
For all $x\in E_{6}(-1,1)$, there exists $\delta (\rho )=4>0$ satisfying%
\begin{equation*}
3\leq M(x,-1)+M(3,1)<3+4\Longrightarrow d(fx,-1)+d(fx,1)=6\leq \rho \text{,}
\end{equation*}%
hence the condition $(a)$ is satisfied.

For $x=12$, we have $d(x,fx)\neq 0$, $M(12,-1)=16,$ $M(12,1)=14$ and so, we
obtain
\begin{eqnarray*}
\zeta \left( d(fx,x),M(x,x_{1})\right) &=&\zeta \left( 6,16\right) =\frac{16%
}{2}-6 \\
&=&2>0
\end{eqnarray*}%
and
\begin{eqnarray*}
\zeta \left( d(fx,x),M(x,x_{2})\right) &=&\zeta \left( 6,14\right) =\frac{14%
}{2}-6 \\
&=&1>0.
\end{eqnarray*}%
This shows that the condition $(b)$ is also satisfied by $f$.

Since we have $d(x,fx)>0$ for $x=12$, we find%
\begin{eqnarray*}
\zeta \left( d(fx,x),\frac{d(x,x_{1})+d(fx,x_{1})+d(fx,x_{2})+d(x,x_{2})}{2}%
\right) &=&\zeta \left( 6,30\right) \\
&=&\frac{30}{2}-6=9>0\text{,}
\end{eqnarray*}%
hence the condition $(c)$ is satisfied.

Clearly, we have $Fix(f)=\left\{ -3,-1,1,3,18\right\} $ and the ellipse $%
E_{6}(-1,1)=\left\{ -3,3\right\} $ is contained in the set $Fix(f)$. That
is, the ellipse $E_{6}(-1,1)$ is a fixed ellipse of the self-mapping $f$.

On the other hand, it is easy to check that the self-mapping $f$ satisfies
the conditions of Theorem \ref{thm22} for the point $x_{0}=3$ and the
simulation function $\zeta (t,s)=\frac{2}{3}s-t$. Clearly, the set $Fix(f)$
contains the circle $C_{3,6}=\left\{ -3\right\} $.
\end{example}

\begin{definition}
\label{def21} Let $\zeta \in \mathcal{Z}$ be any simulation function. The
self-mapping $f$ is said to be a $\mathcal{Z}_{E}$-contraction with respect
to $\zeta $ if there exist $x_{1},x_{2}\in X$ such that the following
condition holds for all $x\in X:$%
\begin{equation*}
d(fx,x)>0\Rightarrow \zeta \left( d(fx,x),d(fx,x_{1})+d(fx,x_{2})\right)
\geq 0.
\end{equation*}
\end{definition}

If $f$ is a $\mathcal{Z}_{E}$-contraction with respect to $\zeta $, then we
have%
\begin{equation}
d(fx,x)<d(fx,x_{1})+d(fx,x_{2}),  \label{eqn1}
\end{equation}%
for all $x\in X$ with $fx\neq x_{1}$ or $fx\neq x_{2}$. Indeed, if $fx=x$
then the inequality (\ref{eqn1}) is satisfied trivially. If $fx\neq x$ then $%
d(fx,x)>0$ and by the definition of a $\mathcal{Z}_{E}$-contraction and the
condition $(\zeta _{2})$, we obtain%
\begin{equation*}
0\leq \zeta \left( d(fx,x),d(fx,x_{1})+d(fx,x_{2})\right)
<d(fx,x_{1})+d(fx,x_{2})-d(fx,x)
\end{equation*}%
and so Equation (\ref{eqn1}) is satisfied.

Now we give the following theorem.

\begin{theorem}
\label{thm1} Let $f$ be a $\mathcal{Z}_{E}$-contraction with respect to $%
\zeta $ with $x_{1},x_{2}\in X$ and consider the set
\begin{equation*}
\overline{E}_{\rho }(x_{1},x_{2})=\left\{ x\in X:d\left( x,x_{1}\right)
+d\left( x,x_{2}\right) \leq \rho \right\} \text{.}
\end{equation*}
If the condition $0<d(fx,x_{1})+d(fx,x_{2})\leq \rho $ holds for all $x\in
\overline{E}_{\rho }(x_{1},x_{2})-\left\{ x_{1},x_{2}\right\} $ then $Fix(f)$
contains the set $\overline{E}_{\rho }(x_{1},x_{2})$.
\end{theorem}

\begin{proof}
If $\rho =0$, then we have $\overline{E}_{\rho }(x_{1},x_{2})=D_{x_{1},\rho
}=\{x_{1}\}$ and this theorem coincides with Theorem 2.2 in \cite%
{Ozgur-simulation}. In this case, we have $fx_{1}=x_{1}$. Assume that $\rho
\neq 0$. If $x_{1}=x_{2}$ then $\overline{E}_{\rho }(x_{1},x_{2})=D_{x_{1},%
\frac{\rho }{2}}$ and again this case is reduced to Theorem 2.2 in \cite%
{Ozgur-simulation}.

Assume that $x_{1}\neq x_{2}$ and let $x\in \overline{E}_{\rho
}(x_{1},x_{2}) $ be such that $fx\neq x$. By the definition of $\rho $, we
have $0<\rho \leq d(x,fx)$ and using the condition $(\zeta _{2})$ we find
\begin{eqnarray*}
\zeta \left( d(fx,x),d(fx,x_{1})+d(fx,x_{2})\right)
&<&d(fx,x_{1})+d(fx,x_{2})-d(fx,x) \\
&\leq &\rho -d(fx,x)\leq \rho -\rho =0\text{,}
\end{eqnarray*}%
a contradiction with the $\mathcal{Z}_{E}$-contractive property of $f$. This
contradiction leads $fx=x$, so the set $Fix(f)$ contains the set $\overline{E%
}_{\rho }(x_{1},x_{2})$.
\end{proof}

\begin{example}
\label{exm22} Let us consider the self-mapping $f$ defined in Example \ref%
{exm21}. $f$ is an $\mathcal{Z}_{E}$-contraction with respect to $\zeta
(t,s)=\frac{1}{2}s-t$ with the points $x_{1}=-1$ and $x_{2}=1$. Indeed, we
get
\begin{eqnarray*}
\zeta \left( d(fx,x),d(fx,x_{1})+d(fx,x_{2})\right) &=&\zeta \left(
6,19+17\right) =\zeta \left( 6,36\right) \\
&=&\frac{36}{2}-6=12\geq 0\text{,}
\end{eqnarray*}%
for $x=12$ with $d(fx,x)>0$. Also we have,%
\begin{equation*}
0<d(fx,-1)+d(fx,1)\leq 6\text{,}
\end{equation*}%
for all $x\in \overline{E}_{6}(-1,1)-\left\{ -1,1\right\} $. Therefore, $f$
satisfies the conditions of Theorem \ref{thm1}. Notice that the set $Fix(f)$
contains the set $\overline{E}_{6}(-1,1)=\left\{ -3,-1,1,3\right\} $.
\end{example}

Let $r\in \left[ 0,\infty \right) $. Now we give a fixed-circle theorem
using the auxiliary function $\varphi _{r}:%
\mathbb{R}
^{+}\cup \left\{ 0\right\} \rightarrow
\mathbb{R}
$ defined by%
\begin{equation}
\varphi _{r}\left( u\right) =\left\{
\begin{array}{ccc}
u-r & ; & u>0 \\
0 & ; & u=0%
\end{array}%
\right. \text{,}  \label{auxiliary function}
\end{equation}%
for all $u\in
\mathbb{R}
^{+}\cup \left\{ 0\right\} $ (see \cite{Ozgur-Aip}).

\begin{theorem}
\label{thm23} Let $(X,d)$ be a metric space, $\zeta \in \mathcal{Z}$ be a
simulation function and $C_{x_{0},r}$ be any circle on $X$. If there exists
a self-mapping $f:X\rightarrow X$ satisfying

$i)$ $d\left( x_{0},fx\right) =r$ for each $x\in C_{x_{0},r}$,

$ii)$ $\zeta \left( r,d\left( fx,fy\right) \right) \geq 0$ for each $x,y\in
C_{x_{0},r}$ with $x\neq y$,

$iii)$ $\zeta \left( d\left( fx,fy\right) ,d\left( x,y\right) -\varphi
_{r}\left( d\left( x,fx\right) \right) \right) \geq 0$ for each $x,y\in
C_{x_{0},r}$,

$iv)$ $f$ is one to one on the circle $C_{x_{0},r}$, \newline
then the circle $C_{x_{0},r}$ is a fixed circle of $f$.
\end{theorem}

\begin{proof}
Let $x\in C_{x_{0},r}$ be an arbitrary point. By the condition $(i)$, we
have $d\left( x_{0},fx\right) =r$, that is, $fx\in C_{x_{0},r}$. Now we show
that $fx=x$ for all $x\in C_{x_{0},r}$. Conversely, assume that $x\neq fx$
for any $x\in C_{x_{0},r}$. Then we have $d\left( x,fx\right) >0$ and using
the conditions $(ii)$, $(iv)$ and $(\zeta _{2})$, we get%
\begin{equation*}
0\leq \zeta \left( r,d\left( fx,f^{2}x\right) \right) <d\left(
fx,f^{2}x\right) -r
\end{equation*}%
and so%
\begin{equation}
r<d\left( fx,f^{2}x\right) .  \label{eqn22}
\end{equation}%
Using the definition of the function $\varphi _{r}$ and the conditions $%
(iii) $, $(iv)$ and $(\zeta _{2})$, we obtain%
\begin{eqnarray*}
0 &\leq &\zeta \left( d\left( fx,f^{2}x\right) ,d\left( x,fx\right) -\varphi
_{r}\left( d\left( x,fx\right) \right) \right) =\zeta \left( d\left(
fx,f^{2}x\right) ,d\left( x,fx\right) -\left( d\left( x,fx\right) -r\right)
\right) \\
&=&\zeta \left( d\left( fx,f^{2}x\right) ,r\right) <r-d\left(
fx,f^{2}x\right)
\end{eqnarray*}%
and hence%
\begin{equation*}
d\left( fx,f^{2}x\right) <r\text{,}
\end{equation*}%
which is a contradiction with the inequality (\ref{eqn22}). Therefore it
should be $fx=x$ for each $x\in C_{x_{0},r}$. Consequently, $C_{x_{0},r}$ is
a fixed circle of $f$.
\end{proof}

\begin{remark}
\label{rem21} If we consider the self-mapping $f$ defined in Example \ref%
{exm21}, it is easy to verify that $f$ satisfies the conditions of Theorem %
\ref{thm21} and Theorem \ref{thm1} for the ellipse $E_{6}(-3,3)=\left\{
-3,-1,1,3\right\} $ with the simulation function $\zeta (t,s)=\frac{2}{3}s-t$%
. This shows that the fixed ellipse is not unique for the number $\rho $
defined in $($\ref{number}$)$. On the other hand, the fixed point set $%
Fix(f)=\left\{ -3,-1,1,3,18\right\} $ contains also the ellipses $%
E_{4}(-3,1)=\left\{ -3,-1,1\right\} $ and $E_{4}(-1,3)=\left\{
-1,1,3\right\} $ other than the ellipses $E_{6}(-3,3)$ and $E_{6}(-1,1)$. We
deduce that the number $\rho $ defined in $($\ref{number}$)$ can not produce
all fixed ellipses $($resp. circles$)$ for a self-mapping $f$.
\end{remark}

This remark shows also that a fixed ellipse may not be unique. Now, we give
a general result which ensure the uniqueness of a fixed geometric figure
(for example, a circle, an Apollonius circle, an ellipse, a hyperbola, etc.)
for a self-mapping of a metric space $(X,d)$.

\begin{theorem}
\label{thm24} $($The uniqueness theorem$)$ Let $(X,d)$ be a metric space,
the number $M\left( x,y\right) $ be defined as in $($\ref{definition of M}$)$
and $f:X\rightarrow X$ be a self-mapping. Assume that the fixed point set $%
Fix(f)$ contains a geometric figure $\mathcal{F}$. If there exists a
simulation function $\zeta \in \mathcal{Z}$ \ such that the condition
\begin{equation}
\zeta \left( d\left( fx,fy\right) ,M\left( x,y\right) \right) \geq 0
\label{eqn23}
\end{equation}%
is satisfied by $f$ for all $x\in \mathcal{F}$ and $y\in X-\mathcal{F}$,
then the figure $\mathcal{F}$ is the unique fixed figure of $f$.
\end{theorem}

\begin{proof}
Assume that $\mathcal{F}^{\ast }$ is another fixed figure of $f$. Let $x\in
\mathcal{F}$, $y\in \mathcal{F}^{\ast }$ with $x\neq y$ be arbitrary points.
Using the inequality (\ref{eqn23}) and the condition $(\zeta _{2})$, we find%
\begin{equation*}
0\leq \zeta \left( d\left( fx,fy\right) ,M\left( x,y\right) \right) =\zeta
\left( d\left( x,y\right) ,d\left( x,y\right) \right) <d\left( x,y\right)
-d\left( x,y\right) =0\text{,}
\end{equation*}%
a contradiction. Hence, it should be $x=y$ for all $x\in \mathcal{F}$, $y\in
\mathcal{F}^{\ast }$. This shows the uniqueness of the fixed figure $%
\mathcal{F}$ of $f$.
\end{proof}

Now we give a condition which excludes the identity map $I_{X}:X\rightarrow
X $ defined by $I_{X}(x)=x$ for all $x\in X$ from the above results.

\begin{theorem}
\label{thm25} Let $(X,d)$ be a metric space, $f:X\rightarrow X$ be a\
self-mapping and $r\in \left[ 0,\infty \right) $ be a fixed number. If there
exists a simulation function $\zeta \in \mathcal{Z}$ such that the condition
\begin{equation*}
d\left( x,fx\right) <\zeta \left( d\left( x,fx\right) ,\varphi _{r}\left(
d\left( x,fx\right) \right) +r\right)
\end{equation*}%
is satisfied by $f$ for all $x\notin Fix(f)$ if and only if $f=I_{X}$.
\end{theorem}

\begin{proof}
Let $x\in X$ be an arbitrary point with $x\notin Fix(f)$. Using (\ref%
{auxiliary function}) and the condition $(\zeta _{2})$, we find%
\begin{eqnarray*}
d\left( x,fx\right) &\leq &\zeta \left( d\left( x,fx\right) ,\varphi
_{r}\left( d\left( x,fx\right) \right) +r\right) =\zeta \left( d\left(
x,fx\right) ,\left( d\left( x,fx\right) -r\right) +r\right) \\
&=&\zeta \left( d\left( x,fx\right) ,d\left( x,fx\right) \right) <d\left(
x,fx\right) -d\left( x,fx\right) =0\text{,}
\end{eqnarray*}%
a contradiction. Hence, it should be $x\in Fix(f)$ for all $x\in X$, that
is, $Fix(f)=X$. This shows that $f=I_{X}$. Clearly, the identity map $I_{X}$
satisfies the condition of the hypothesis for any simulation function $\zeta
\in \mathcal{Z}$.
\end{proof}

\subsection{\textbf{Geometric study of fixed points on $S$-metric and $b$%
-metric spaces}}

At first, we recall the concept of an $S$-metric space.

\begin{definition}
\label{def31} \cite{Sedghi 2012} Let $X$ be nonempty set and $\mathcal{S}%
:X^{3}\rightarrow \lbrack 0,\infty )$ be a function satisfying the following
conditions

\begin{enumerate}
\item $\mathcal{S}(x,y,z)=0$ if and only if $x=y=z$,

\item $\mathcal{S}(x,y,z)\leq \mathcal{S}(x,x,a)+\mathcal{S}(y,y,a)+\mathcal{%
S}(z,z,a)$,
\end{enumerate}

for all $x,y,z,a\in X$. Then $S$ is called an $S$-metric on $X$ and the pair
$(X,\mathcal{S})$ is called an $S$-metric space.
\end{definition}

Let $(X,d)$ be a metric space. It is known that the function $\mathcal{S}%
_{d}:X^{3}\rightarrow \left[ 0,\infty \right) $ defined by
\begin{equation*}
\mathcal{S}_{d}\left( x,y,z\right) =d\left( x,z\right) +d\left( y,z\right)
\end{equation*}%
for all $x,y,z\in X$ is an $S$-metric on $X$ \cite{Hieu 2015}. The $S$%
-metric $\mathcal{S}_{d}$ is called the $S$-metric generated by the metric $%
d $ \cite{Ozgur 2017}. For example, let $X=%
\mathbb{R}
$ and the function $\mathcal{S}:X^{3}\rightarrow \lbrack 0,\infty )$ be
defined by%
\begin{equation}
\mathcal{S}(x,y,z)=\left\vert x-z\right\vert +\left\vert y-z\right\vert
\text{,}  \label{usual S metric}
\end{equation}%
for all $x,y,z\in
\mathbb{R}
$ \cite{Sedghi 2014}. Then $(X,\mathcal{S})$ is called the usual $S$-metric
space. This $S$-metric is generated by the usual metric on $%
\mathbb{R}
$. The main motivation of this subsection is the existence of some examples
of $S$-metrics which are not generated by any metric. For example, let $X=%
\mathbb{R}
$ and the function $\mathcal{S}:X^{3}\rightarrow \lbrack 0,\infty )$ be
defined by%
\begin{equation}
\mathcal{S}(x,y,z)=\left\vert x-z\right\vert +\left\vert x+z-2y\right\vert
\text{,}  \label{second S metric}
\end{equation}%
for all $x$, $y$ and $z\in
\mathbb{R}
$. Then, $\mathcal{S}$ is an $S$-metric on $%
\mathbb{R}
$, which is not generated by any metric, and the pair $\left(
\mathbb{R}
,\mathcal{S}\right) $ is an $S$-metric space (see \cite{Ozgur 2017} for more
details and examples).

Let $(X,\mathcal{S})$ be an $S$-metric space and $f:X\rightarrow X$ be a
self-mapping. In this subsection, we give new solutions to the fixed-circle
problem (resp. fixed-disc problem and fixed ellipse problem) for
self-mappings of an $S$-metric space (resp. a $b$-metric space). For the $S$%
-metric case, we use the following numbers%
\begin{equation}
\mu =\inf \left\{ \mathcal{S}(x,x,fx):x\neq fx,x\in X\right\} \text{,}
\label{radius}
\end{equation}%
\begin{equation}
M_{S}(x,y)=\max \left\{
\begin{array}{c}
a\mathcal{S}(x,x,fx)+(1-a)\mathcal{S}(y,y,fy), \\
(1-a)\mathcal{S}(x,x,fx)+a\mathcal{S}(y,y,fy),\frac{\mathcal{S}(x,x,fy)+%
\mathcal{S}(y,y,fx)}{4}%
\end{array}%
\right\} ,\text{ }0\leq a<1  \label{the number Ms(x,y)}
\end{equation}%
and the following symmetry property given in \cite{Sedghi 2012}
\begin{equation}
\mathcal{S}\left( x,x,y\right) =\mathcal{S}\left( y,y,x\right) \text{,}
\label{symmetry property}
\end{equation}%
for all $x,y\in X$ on an $S$-metric space $(X,\mathcal{S})$. Before stating
our results, we recall the definitions of a circle, a disc and an ellipse on
an $S$-metric space, respectively, as follows:%
\begin{equation*}
C_{x_{0},r}^{S}=\left\{ x\in X:\mathcal{S}(x,x,x_{0})=r\right\} \text{,}
\end{equation*}%
\begin{equation*}
D_{x_{0},r}^{S}=\left\{ x\in X:\mathcal{S}(x,x,x_{0})\leq r\right\}
\end{equation*}%
and%
\begin{equation*}
E_{r}^{S}(x_{1},x_{2})=\left\{ x\in X:\mathcal{S}(x,x,x_{1})+\mathcal{S}%
(x,x,x_{2})=r\right\} \text{,}
\end{equation*}%
where $r\in \lbrack 0,\infty )$ \cite{Ozgur-Tas-circle-thesis}, \cite{Sedghi
2012}. For a self-mapping $f$ of an $S$-metric space, the definition of a
fixed figure (circle, disc, ellipse, etc.) can be given similar to the case
introduced in the previous section (see \cite{Ozgur-Tas-circle-thesis} and
\cite{Mlaiki 2018} for the definitions of a fixed circle and a fixed disc).

\begin{theorem}
\label{thms1} Let $(X,\mathcal{S})$ be an $S$-metric space, $f:X\rightarrow
X $ be a self-mapping, $\zeta \in \mathcal{Z}$ be a simulation function and
the number $\mu $ be defined as in $($\ref{radius}$)$. If there exist some
points $x_{1},x_{2}\in X$ such that

$(a)$ For all $x\in E_{\rho }^{S}(x_{1},x_{2})$, there exists $\delta (\mu
)>0$ satisfying%
\begin{equation*}
\frac{\mu }{2}\leq M_{S}(x,x_{1})+M_{S}(x,x_{2})<\frac{\mu }{2}+\delta (\mu
)\Longrightarrow \mathcal{S}(fx,fx,x_{1})+\mathcal{S}(fx,fx,x_{2})\leq \mu
\text{,}
\end{equation*}

$(b)$ For all $x\in X$,%
\begin{equation*}
\mathcal{S}(fx,fx,x)>0\Longrightarrow \zeta \left( \mathcal{S}%
(fx,fx,x),M_{S}(x,x_{1})\right) \geq 0\text{ and }\zeta \left( \mathcal{S}%
(fx,fx,x),M_{S}(x,x_{2})\right) \geq 0\text{,}
\end{equation*}

$(c)$ For all $x\in X$,%
\begin{equation*}
\mathcal{S}(fx,fx,x)>0\Longrightarrow \zeta \left( \mathcal{S}(fx,fx,x),%
\frac{\mathcal{S}(x,x,x_{1})+\mathcal{S}(fx,fx,x_{1})+\mathcal{S}(x,x,x_{2})+%
\mathcal{S}(fx,fx,x_{2})}{2}\right) \geq 0\text{,}
\end{equation*}%
then $fx_{1}=x_{1}$, $fx_{2}=x_{2}$ and $Fix(f)$ contains the ellipse $%
E_{\mu }^{S}(x_{1},x_{2})$.
\end{theorem}

\begin{proof}
We have
\begin{equation*}
M_{S}(x_{1},x_{1})=\mathcal{S}(x_{1},x_{1},fx_{1})\text{ and }%
M_{S}(x_{2},x_{2})=\mathcal{S}(x_{2},x_{2},fx_{2})\text{.}
\end{equation*}%
First, we show that $fx_{1}=x_{1}$ and $fx_{2}=x_{2}$. If $fx_{1}\neq x_{1}$
and $fx_{2}\neq x_{2}$ then $\mathcal{S}(x_{1},x_{1},fx_{1})>0$ and $%
\mathcal{S}(x_{2},x_{2},fx_{2})>0$. Using the symmetry condition (\ref%
{symmetry property}) and the condition $(\zeta _{2})$, we obtain
\begin{eqnarray*}
\zeta \left( \mathcal{S}(fx_{1},fx_{1},x_{1}),M_{S}(x_{1},x_{1})\right)
&=&\zeta \left( \mathcal{S}(fx_{1},fx_{1},x_{1}),\mathcal{S}%
(x_{1},x_{1},fx_{1})\right) \\
&<&\mathcal{S}(fx_{1},fx_{1},x_{1})-\mathcal{S}(x_{1},x_{1},fx_{1})=0
\end{eqnarray*}%
and%
\begin{eqnarray*}
\zeta \left( \mathcal{S}(fx_{2},fx_{2},x_{2}),M_{S}(x_{2},x_{2})\right)
&=&\zeta \left( \mathcal{S}(fx_{2},fx_{2},x_{2}),d(fx_{2},x_{2})\right) \\
&<&d(fx_{2},x_{2})-d(fx_{2},x_{2})=0\text{,}
\end{eqnarray*}%
which are contradictions by the condition $(b)$. Hence it should be $%
fx_{1}=x_{1}$ and $fx_{2}=x_{2}$.

If $\mu =0$, it is easy to check that $E_{\mu
}^{S}(x_{1},x_{2})=C_{x_{1},\mu }^{S}=\{x_{1}\}$ \ and $x_{1}=x_{2}$. Hence,
the proof is completed.

Assume that $\mu \neq 0$. Let $x\in E_{\mu }^{S}(x_{1},x_{2})$ be any point
such that $fx\neq x$. Then $\mathcal{S}(x,x,fx)>0$ and we have%
\begin{equation*}
M^{S}(x,x_{1})=\max \left\{ a\mathcal{S}(x,x,fx),(1-a)\mathcal{S}(x,x,fx),%
\frac{\mathcal{S}(x,x,fx_{1})+\mathcal{S}(x_{1},x_{1},fx)}{4}\right\}
\end{equation*}%
and
\begin{equation*}
M^{S}(x,x_{2})=\max \left\{ a\mathcal{S}(x,x,fx),(1-a)\mathcal{S}(x,x,fx),%
\frac{\mathcal{S}(x,x,fx_{2})+\mathcal{S}(x_{2},x_{2},fx)}{4}\right\} \text{.%
}
\end{equation*}%
Using the condition $(a)$, we get%
\begin{equation}
\frac{\mu }{2}\leq M^{S}(x,x_{1})+M^{S}(x,x_{2})<\frac{\mu }{2}+\delta (\mu
)\Longrightarrow \mathcal{S}(fx,fx,x_{1})+\mathcal{S}(fx,fx,x_{2})\leq \mu
\text{.}  \label{eqn21s}
\end{equation}%
Now, using the inequality (\ref{eqn21s}), the conditions $(c)$, $(\zeta
_{2}) $ and the symmetry condition (\ref{symmetry property}), we obtain%
\begin{eqnarray*}
0 &\leq &\zeta \left( \mathcal{S}(fx,fx,x),\frac{\mathcal{S}(x,x,x_{1})+%
\mathcal{S}(fx,fx,x_{1})+\mathcal{S}(x,x,x_{2})+\mathcal{S}(fx,fx,x_{2})}{2}%
\right) \\
&<&\frac{\mathcal{S}(x,x,x_{1})+\mathcal{S}(fx,fx,x_{1})+\mathcal{S}%
(x,x,x_{2})+\mathcal{S}(fx,fx,x_{2})}{2}-\mathcal{S}(fx,fx,x) \\
&=&\frac{\mathcal{S}(x,x,x_{1})+\mathcal{S}(x,x,x_{2})}{2}+\frac{\mathcal{S}%
(fx,fx,x_{1})+\mathcal{S}(fx,fx,x_{2})}{2}-\mathcal{S}(fx,fx,x) \\
&\leq &\frac{\mu }{2}+\frac{\mu }{2}-\mathcal{S}(fx,fx,x)=\mu -\mathcal{S}%
(fx,fx,x)
\end{eqnarray*}%
and so%
\begin{equation*}
\mathcal{S}(fx,fx,x)<\mu \text{,}
\end{equation*}%
which is a contradiction by the definition of the number $\mu $. This
contradiction leads to $fx=x$. Consequently, we have $E_{\mu
}^{S}(x_{1},x_{2})\subset Fix(f)$.
\end{proof}

If $x_{1}=x_{2}$ then we have $E_{\mu }^{S}(x_{1},x_{2})=C_{x_{1},\frac{\mu
}{2}}$ and Theorem \ref{thms1} is reduced to a fixed-circle theorem as
follows$:$

\begin{theorem}
\label{thm4} Let $(X,\mathcal{S})$ be an $S$-metric space, $f:X\rightarrow X$
be a self-mapping and the number $\mu $ be defined as in $($\ref{radius}$)$.
If there exist a simulation function $\zeta \in \mathcal{Z}$ and a point $%
x_{0}\in X$ such that

$(i)$ For all $x\in C_{x_{0},\mu }^{S}$, there exists a $\delta (\mu )>0$
satisfying%
\begin{equation*}
\frac{\mu }{4}\leq M_{S}(x,x_{0})<\frac{\mu }{4}+\delta (\mu
)\Longrightarrow \mathcal{S}(fx,fx,x_{0})\leq \mu ,
\end{equation*}

$(ii)$ For all $x\in X$,%
\begin{equation*}
\mathcal{S}(fx,fx,x)>0\Longrightarrow \zeta (\mathcal{S}\left(
fx,fx,x\right) ,M_{S}(x,x_{0}))\geq 0,
\end{equation*}%
then $x_{0}\in Fix\left( f\right) $ and the circle $C_{x_{0},\mu }^{S}$ is a
fixed circle of $f$.
\end{theorem}

\begin{proof}
If $fx_{0}\neq x_{0}$ then using the symmetry property (\ref{symmetry
property}) and the condition $(\zeta _{2})$, we get%
\begin{equation*}
M_{S}(x_{0},x_{0})=\mathcal{S}\left( x_{0},x_{0},fx_{0}\right) =\mathcal{S}%
\left( fx_{0},fx_{0},x_{0}\right)
\end{equation*}%
and
\begin{eqnarray*}
\zeta \left( \mathcal{S}\left( fx_{0},fx_{0},x_{0}\right)
,M_{S}(x_{0},x_{0})\right) &=&\zeta \left( \mathcal{S}\left(
fx_{0},fx_{0},x_{0}\right) ,\mathcal{S}\left( fx_{0},fx_{0},x_{0}\right)
\right) \\
&<&\mathcal{S}\left( fx_{0},fx_{0},x_{0}\right) -\mathcal{S}\left(
fx_{0},fx_{0},x_{0}\right) =0\text{.}
\end{eqnarray*}%
This last inequality is a contradiction by $(ii)$. Therefore, it should be $%
fx_{0}=x_{0}$. This shows that the circle $C_{x_{0},\mu }^{S}=\{x_{0}\}$ is
a fixed circle of $f$ when $\mu =0$. Now, let $\mu >0$ and $x\in
C_{x_{0},\mu }^{S}$ be any element. To show that $f$ fixes the circle $%
C_{x_{0},\mu }^{S}$, we suppose that $fx\neq x$ for any $x\in C_{x_{0},\mu
}^{S}$. Since $x_{0}\in Fix\left( f\right) $, we have
\begin{eqnarray*}
M_{S}(x,x_{0}) &=&\max \left\{ a\mathcal{S}(x,x,fx),(1-a)\mathcal{S}(x,x,fx),%
\frac{\mathcal{S}(x,x,fx_{0})+\mathcal{S}(x_{0},x_{0},fx)}{4}\right\} \\
&=&\max \left\{ a\mathcal{S}(x,x,fx),(1-a)\mathcal{S}(x,x,fx),\frac{\mu +%
\mathcal{S}(x_{0},x_{0},fx)}{4}\right\} .
\end{eqnarray*}%
Using the conditions $(i)$, $(ii)$, $(\zeta _{2})$ and the symmetry property
(\ref{symmetry property}), we find%
\begin{eqnarray*}
0 &<&\zeta \left( \mathcal{S}\left( fx,fx,x\right) ,M_{S}(x,x_{0})\right) \\
&<&M_{S}(x,x_{0})-\mathcal{S}\left( fx,fx,x\right) <\mu -\mathcal{S}\left(
fx,fx,x\right) \text{,}
\end{eqnarray*}%
a contradiction by the definition of the number $\mu $. Hence, we have $fx=x$%
. Consequently, $f$ fixes the circle $C_{x_{0},\mu }^{S}$.
\end{proof}

\begin{corollary}
\label{cor1} Let $(X,\mathcal{S})$ be an $S$-metric space, $f:X\rightarrow X$
be a self-mapping and the number $\mu $ be defined as in $($\ref{radius}$)$.
If there exist a simulation function $\zeta \in \mathcal{Z}$ and a point $%
x_{0}\in X$ such that

$(i)$ For all $x\in D_{x_{0},\mu }^{S}$, there exists a $\delta (\mu )>0$
satisfying%
\begin{equation*}
\frac{\mu }{4}\leq M_{S}(x,x_{0})<\frac{r}{4}+\delta (\mu )\Longrightarrow
\mathcal{S}(fx,fx,x_{0})\leq \mu ,
\end{equation*}

$(ii)$ For all $x\in X$,%
\begin{equation*}
\mathcal{S}(fx,fx,x)>0\Longrightarrow \zeta (\mathcal{S}\left(
fx,fx,x\right) ,M_{S}(x,x_{0}))\geq 0,
\end{equation*}%
then the disc $D_{x_{0},\mu }^{S}$ is a fixed disc of $f$.
\end{corollary}

\begin{example}
\label{exm23} Let $X=\left\{ -3,-1,1,3,12,18\right\} $ with the $S$-metric
defined in $($\ref{second S metric}$)$ and consider the self-mapping $f$
defined in Example \ref{exm21} on this $S$-metric space $(X,\mathcal{S})$.
It is easy to check that the self-mapping $f$ satisfies the conditions of
Theorem \ref{thms1} with the points $x_{1}=-1$ and $x_{2}=1$, the simulation
function $\zeta (t,s)=\frac{7}{8}s-t$ and any $a\in \left[ 0,1\right) $. We
have%
\begin{equation*}
\mu =\inf \left\{ \mathcal{S}(x,x,fx):x\neq fx,x\in X\right\} =\mathcal{S}%
(12,12,18)=12
\end{equation*}%
and clearly the set $Fix(f)$ contains the ellipse $E_{12}^{S}(-1,1)=\left\{
-3,3\right\} $. On the other hand, the self-mapping $f$ does not satisfy the
condition $(b)$ of Theorem \ref{thms1} for the ellipse $E_{12}^{S}(-3,3)=%
\left\{ -3,-1,1,3\right\} $ for any simulation function $\zeta $ and any $%
a\in \left[ 0,1\right) $. Indeed, we have
\begin{eqnarray*}
\zeta \left( \mathcal{S}(fx,fx,x),M_{S}(x,3)\right) &=&\zeta (\mathcal{S}%
\left( 18,18,12\right) ,M_{S}(12,3)) \\
&=&\zeta (12,12))<0\text{,}
\end{eqnarray*}%
by the condition $(\zeta _{2})$ for the point $x=12$ with $\mathcal{S}%
(18,18,12)>0$. This shows that the converse statement of Theorem \ref{thms1}
is not true everwhen.

The self-mapping $f$ satisfies the conditions of Theorem \ref{thm4} with the
point $x_{0}=-3$ and the simulation function $\zeta (t,s)=\frac{5}{6}s-t$.
The circle $C_{-3,12}^{S}=\left\{ 3\right\} $ is contained in the set $%
Fix(f) $. On the other hand, the self-mapping $f$ does not satisfy the
condition $(ii)$ of Theorem \ref{thm4} with the point $x_{0}=3$ for any
simulation function $\zeta $ and any $a\in \left[ 0,1\right) $. Since we
have
\begin{eqnarray*}
\zeta (\mathcal{S}\left( fx,fx,x\right) ,M_{S}(x,x_{0})) &=&\zeta (\mathcal{S%
}\left( 18,18,12\right) ,M_{S}(12,3)) \\
&=&\zeta (12,12))<0\text{,}
\end{eqnarray*}%
by the condition $(\zeta _{2})$ for the point $x=12$. But, the circle $%
C_{3,12}^{S}=\left\{ -3\right\} $ is a fixed circle of $f$. This shows that
the converse statement of Theorem \ref{thm4} is not true everwhen.
\end{example}

\begin{remark}
$1)$ Example \ref{exm23} shows the importance of the studies on an $S$%
-metric space. Notice that the ellipse $E_{12}(-1,1)$ and the circle $%
C_{-3,12}$ are empty sets in the metric space $\left( X,d\right) $ in
Example \ref{exm21} but, if we consider the $S$-metric $\mathcal{S}%
(x,y,z)=\left\vert x-z\right\vert +\left\vert x+z-2y\right\vert $ on $X$
then the ellipse $E_{12}^{S}(-1,1)$ and the circle $C_{-3,12}^{S}$ are not
empty sets and both of them are contained in the set $Fix(f)$.

$2)$ Note that $S$-metric versions of Definition \ref{def21} and Theorem \ref%
{thm1} can also be given for self-mappings on an $S$-metric space.
\end{remark}

Now we give a fixed-circle theorem using the auxiliary function $\varphi
_{r}:%
\mathbb{R}
^{+}\cup \left\{ 0\right\} \rightarrow
\mathbb{R}
$ defined in (\ref{auxiliary function}).

\begin{theorem}
\label{thm26} $(X,\mathcal{S})$ be an $S$-metric space, $\zeta \in \mathcal{Z%
}$ be a simulation function and $C_{x_{0},r}^{S}$ be any circle on $X$ with $%
r>0$. If there exists a self-mapping $f:X\rightarrow X$ satisfying

$i)$ $\mathcal{S}(fx,fx,x_{0})=r$ for each $x\in C_{x_{0},r}^{S}$,

$ii)$ $\zeta \left( r,\mathcal{S}\left( fx,fx,fy\right) \right) \geq 0$ for
each $x,y\in C_{x_{0},r}^{S}$ with $x\neq y$,

$iii)$ $\zeta \left( \mathcal{S}\left( fx,fx,fy\right) ,\mathcal{S}\left(
x,x,y\right) -\varphi _{r}\left( \mathcal{S}\left( x,x,fx\right) \right)
\right) \geq 0$ for each $x,y\in C_{x_{0},r}^{S}$,

$iv)$ $f$ is one to one on the circle $C_{x_{0},r}^{S}$,

then the circle $C_{x_{0},r}^{S}$ is a fixed circle of $f$.
\end{theorem}

\begin{proof}
Let $x\in C_{x_{0},r}^{S}$ be an arbitrary point. By the condition $(i)$, we
have $fx\in C_{x_{0},r}^{S}$. To show that $fx=x$ for all $x\in
C_{x_{0},r}^{S}$, conversely, we assume that $x\neq fx$ for any $x\in
C_{x_{0},r}^{S}$. Then we have $\mathcal{S}\left( x,x,fx\right) >0$. Using
the conditions $(ii)$, $(iv)$ and $(\zeta _{2})$, we obtain%
\begin{equation*}
0\leq \zeta \left( r,\mathcal{S}\left( fx,fx,f^{2}x\right) \right) <\mathcal{%
S}\left( fx,fx,f^{2}x\right) -r
\end{equation*}%
and hence%
\begin{equation}
r<\mathcal{S}\left( fx,fx,f^{2}x\right) .  \label{eqn24}
\end{equation}%
Using the definition of the function $\varphi _{r}$ and the conditions $%
(iii) $, $(iv)$ and $(\zeta _{2})$, we get%
\begin{eqnarray*}
0 &\leq &\zeta \left( \mathcal{S}\left( fx,fx,f^{2}x\right) ,\mathcal{S}%
\left( x,x,fx\right) -\varphi _{r}\left( \mathcal{S}\left( x,x,fx\right)
\right) \right) \\
&=&\zeta \left( \mathcal{S}\left( fx,fx,f^{2}x\right) ,\mathcal{S}\left(
x,x,fx\right) -\left( \mathcal{S}\left( x,x,fx\right) -r\right) \right) \\
&=&\zeta \left( \mathcal{S}\left( fx,fx,f^{2}x\right) ,r\right) <r-\mathcal{S%
}\left( fx,fx,f^{2}x\right)
\end{eqnarray*}%
and hence%
\begin{equation*}
\mathcal{S}\left( fx,fx,f^{2}x\right) <r\text{,}
\end{equation*}%
which is a contradiction with the inequality (\ref{eqn24}). Consequently, it
should be $fx=x$ for each $x\in C_{x_{0},r}^{S}$, that is, $C_{x_{0},r}^{S}$
is a fixed circle of $f$.
\end{proof}

Now we give a general result which ensure the uniqueness of a geometric
figure contained in the set $Fix(f)$ for a self-mapping of an $S$-metric
space $(X,\mathcal{S})$.

\begin{theorem}
\label{thm27} $($The uniqueness theorem$)$ Let $(X,\mathcal{S})$ be an $S$%
-metric space, the number $M_{S}\left( x,y\right) $ be defined as in $($\ref%
{the number Ms(x,y)}$)$ and $f:X\rightarrow X$ be a self-mapping. Assume
that the fixed point set $Fix(f)$ contains a geometric figure $\mathcal{F}$.
If there exists a simulation function $\zeta \in \mathcal{Z}$ \ such that
the condition
\begin{equation}
\zeta \left( \mathcal{S}\left( fx,fx,fy\right) ,M_{S}\left( x,y\right)
\right) \geq 0  \label{eqn25}
\end{equation}%
is satisfied by $f$ for all $x\in \mathcal{F}$ and $y\in X-\mathcal{F}$,
then the figure $\mathcal{F}$ is the unique fixed figure of $f$.
\end{theorem}

\begin{proof}
On the contrary, we suppose that $\mathcal{F}^{\ast }$ is another fixed
figure of the self-mapping $f$. Let $x\in \mathcal{F}$, $y\in \mathcal{F}%
^{\ast }$ with $x\neq y$ be arbitrary points. Using the inequality (\ref%
{eqn25}) and the condition $(\zeta _{2})$, we get%
\begin{equation*}
0\leq \zeta \left( \mathcal{S}\left( fx,fx,fy\right) ,M_{S}\left( x,y\right)
\right) =\zeta \left( \mathcal{S}\left( x,x,y\right) ,\mathcal{S}\left(
x,x,y\right) \right) <\mathcal{S}\left( x,x,y\right) -\mathcal{S}\left(
x,x,y\right) =0\text{,}
\end{equation*}%
a contradiction. Hence, it should be $x=y$ for all $x\in \mathcal{F}$, $y\in
\mathcal{F}^{\ast }$. This shows the uniqueness of the fixed figure $%
\mathcal{F}$ of $f$.
\end{proof}

We give a condition which excludes the identity map $I_{X}:X\rightarrow X$
defined by $I_{X}(x)=x$ for all $x\in X$ from the above results.

\begin{theorem}
\label{thm28} Let $(X,\mathcal{S})$ be an $S$-metric space, $f:X\rightarrow
X $ be a\ self-mapping and $r\in \left[ 0,\infty \right) $ be a fixed
number. If there exists a simulation function $\zeta \in \mathcal{Z}$ \ such
that the condition
\begin{equation*}
\mathcal{S}\left( x,x,fx\right) <\zeta \left( \mathcal{S}\left(
x,x,fx\right) ,\varphi _{r}\left( \mathcal{S}\left( x,x,fx\right) \right)
+r\right)
\end{equation*}%
is satisfied by $f$ for all $x\notin Fix(f)$ if and only if $f=I_{X}$.
\end{theorem}

\begin{proof}
The proof is similar to the proof of Theorem \ref{thm25}.
\end{proof}

\begin{remark}
$1)$ Let $(X,\mathcal{S})$ be an $S$-metric space. Suppose that the $S$%
-metric\ $\mathcal{S}$ is generated by a metric $d$. Then, for $0\leq a<1$
we have%
\begin{eqnarray*}
M_{S}(x,x_{0}) &=&\max \left\{
\begin{array}{c}
a\mathcal{S}(x,x,fx)+(1-a)\mathcal{S}(x_{0},x_{0},fx_{0}), \\
(1-a)\mathcal{S}(x,x,fx)+a\mathcal{S}(x_{0},x_{0},fx_{0}),\frac{\mathcal{S}%
(x,x,fx_{0})+\mathcal{S}(x_{0},x_{0},fx)}{4}%
\end{array}%
\right\} \\
&=&\max \left\{
\begin{array}{c}
2ad(x,fx)+2(1-a)d(x_{0},fx_{0}), \\
2(1-a)d(x,fx)+2ad(x_{0},fx_{0}),\frac{d(x,fx_{0})+d(x_{0},fx)}{2}%
\end{array}%
\right\}
\end{eqnarray*}%
and so
\begin{equation*}
M(x,x_{0})\leq M_{S}(x,x_{0})\text{.}
\end{equation*}%
Consequently, Theorem \ref{thm4} $($resp. Corollary \ref{cor1}$)$ is a
generalization of Theorem $3.2$ $($resp. Corollary $3.2)$ given in \cite%
{Pant 2020}.

$2)$ Similar definition of the notion of a fixed figure $($circle, disc,
ellipse and so on$)$ can be given for a self-mapping of a $b$-metric space.

$3)$ From \cite{Sedghi 2014}, we know that the function defined by
\begin{equation*}
d^{S}\left( x,y\right) =S\left( x,x,y\right) =2d\left( x,y\right) \text{,}
\end{equation*}%
for all $x,y\in X$, defines a $b$-metric on an $S$-metric space $(X,\mathcal{%
S})$ with $b=\frac{3}{2}$. If we consider Theorem \ref{thm4} and Theorem \ref%
{thm26} together with this fact, then we get the following fixed-circle $($%
resp. fixed-disc$)$ results on a $b$-metric space using the number%
\begin{equation}
M_{d^{S}}(x,y)=\max \left\{
\begin{array}{c}
ad^{S}(x,fx)+(1-a)d^{S}(y,fy), \\
(1-a)d^{S}(x,fx)+a(1-a)d^{S}(y,fy),\frac{d^{S}(x,fy)+d^{S}(y,fx)}{2}%
\end{array}%
\right\} \text{, }0\leq a<1\text{.}  \label{the number Mds(x,y)}
\end{equation}
\end{remark}

\begin{theorem}
\label{thm5} Let $(X,d^{S})$ be a $b$-metric space, $f:X\rightarrow X$ a
self-mapping and $\mu $ be defined as
\begin{equation}
\mu =\inf \left\{ d^{S}(x,fx):x\neq fx,x\in X\right\} \text{.}
\label{radius2}
\end{equation}%
If there exist a simulation function $\zeta \in \mathcal{Z}$ and a point $%
x_{0}\in X$ such that

$(i)$ For all $x\in C_{x_{0},\mu }^{d^{S}}$, there exists a $\delta (\mu )>0$
satisfying%
\begin{equation*}
\frac{\mu }{4}\leq M_{d^{S}}(x,x_{0})<\frac{\mu }{4}+\delta (\mu
)\Longrightarrow d^{S}(fx,x_{0})\leq \mu \text{,}
\end{equation*}%
where
\begin{equation*}
C_{x_{0},\mu }^{d^{S}}=\left\{ x\in X:d^{S}(x,x_{0})=\mu \right\}
\end{equation*}%
and%
\begin{equation*}
M_{d^{S}}(x,x_{0})=\max \left\{
\begin{array}{c}
ad^{S}(x,fx)+(1-a)d^{S}(x_{0},fx_{0}), \\
(1-a)d^{S}(x,fx)+ad^{S}(x_{0},fx_{0}),\frac{d^{S}(x,fx_{0})+d^{S}(x_{0},fx)}{%
4}%
\end{array}%
\right\} ,\text{ }0\leq a<1,
\end{equation*}

$(ii)$ For all $x\in X,$
\begin{equation*}
d^{S}(fx,x)>0\Longrightarrow \zeta \left(
d^{S}(fx,x),M_{d^{S}}(x,x_{0})\right) \geq 0,
\end{equation*}%
then $x_{0}\in Fix(f)$ and the circle $C_{x_{0},\mu }^{d^{S}}$ is a fixed
circle of $f$.
\end{theorem}

\begin{corollary}
\label{cor2} Let $(X,d^{S})$ be a $b$-metric space, $f:X\rightarrow X$ be a
self-mapping and the number $\mu $ be defined as in $($\ref{radius2}$)$. If
there exist a simulation function $\zeta \in \mathcal{Z}$ and a point $%
x_{0}\in X$ such that

$(i)$ For all $x\in D_{x_{0},\mu }^{d^{S}}$, there exists a $\delta (\mu )>0$
satisfying%
\begin{equation*}
\frac{\mu }{4}\leq M_{d^{S}}(x,x_{0})<\frac{\mu }{4}+\delta (\mu
)\Longrightarrow d^{S}(fx,x_{0})\leq \mu \text{,}
\end{equation*}%
where%
\begin{equation*}
D_{x_{0},\mu }^{d^{S}}=\left\{ x\in X:d^{S}(x,x_{0})\leq \mu \right\} \text{,%
}
\end{equation*}

$(ii)$ For all $x\in X$,%
\begin{equation*}
d^{S}(fx,x)>0\Longrightarrow \zeta \left(
d^{S}(fx,x),M_{d^{S}}(x,x_{0})\right) \geq 0,
\end{equation*}%
then the disc $D_{x_{0},\mu }^{d^{S}}$ is a fixed disc of $f$.
\end{corollary}

\begin{theorem}
\label{thm29} Let $(X,d^{S})$ be a $b$-metric space, $\zeta \in \mathcal{Z}$
be a simulation function and $C_{x_{0},r}^{d^{S}}$ be any circle on $X$ with
$r>0$. If there exists a self-mapping $f:X\rightarrow X$ satisfying

$i)$ $d^{S}(fx,x_{0})=r$ for each $x\in C_{x_{0},r}^{d^{S}}$,

$ii)$ $\zeta \left( r,d^{S}(fx,fy\right) )\geq 0$ for each $x,y\in
C_{x_{0},r}^{d^{S}}$ with $x\neq y$,

$iii)$ $\zeta \left( d^{S}(fx,fy),d^{S}(x,y)-\varphi _{r}\left(
d^{S}(x,fx\right) \right) )\geq 0$ for each $x,y\in C_{x_{0},r}^{d^{S}}$,

$iv)$ $f$ is one to one on the circle $C_{x_{0},r}^{d^{S}}$,

then the circle $C_{x_{0},r}^{d^{S}}$ is a fixed circle of $f$.
\end{theorem}

\begin{theorem}
\label{thm30} Let $(X,d^{S})$ be a $b$-metric space, the number $%
M_{d^{S}}(x,y)$ be defined as in $($\ref{the number Mds(x,y)}$)$ and $%
f:X\rightarrow X$ be a self-mapping. Assume that the fixed point set $Fix(f)$
contains a geometric figure $\mathcal{F}$. If there exists a simulation
function $\zeta \in \mathcal{Z}$ \ such that the condition
\begin{equation*}
\zeta \left( d^{S}(fx,fy),M_{d^{S}}\left( x,y\right) \right) \geq 0
\end{equation*}%
is satisfied by $f$ for all $x\in \mathcal{F}$ and $y\in X-\mathcal{F}$,
then the figure $\mathcal{F}$ is the unique fixed figure of $f$.
\end{theorem}

\begin{theorem}
\label{thm31} Let $(X,d^{S})$ be a $b$-metric space, $f:X\rightarrow X$ be
a\ self-mapping with the fixed point set $Fix(f)$ and $r\in \left[ 0,\infty
\right) $ be a fixed number. If there exists a simulation function $\zeta
\in \mathcal{Z}$ \ such that the condition
\begin{equation*}
d^{S}(x,fx)<\zeta \left( d^{S}(x,fx),\varphi _{r}\left( d^{S}(x,fx)\right)
+r\right)
\end{equation*}%
is satisfied by $f$ for all $x\notin Fix(f)$ if and only if $f=I_{X}$.
\end{theorem}

\section{\textbf{Conclusions and Future Works}}

In this paper, we have obtained new results on the study of geometric
properties of the fixed point set of a self-mapping on a metric (resp. $S$%
-metric,$\ b$-metric) space via the properties of the set of simulation
functions. As a future work, the determination of new conditions which
ensure a geometric figure to be fixed by a self-mapping can be considered
using similar approaches. Further possible applications of our theoretic
results can be done on the applied sciences using the geometric properties
of fixed points. For example, in \cite{Li}, the existence of a fixed point
for every recurrent neural network was shown using Brouwer's Fixed Point
Theorem and a geometric approach was used to locate where the fixed points
are (see \cite{Li} for more details). Therefore, theoretic fixed figure
results are important in the study of neural networks.




\begin{thebibliography}{99}
\bibitem{Aydi 2019} H. Aydi, N. Ta\c{s}, N. Y. \"{O}zg\"{u}r and N. Mlaiki,
Fixed-discs in rectangular metric spaces, Symmetry 11 (2019), no. 2, 294.

\bibitem{Chanda 2019} A. Chanda, D. Dey, K. Lakshmi and S. Radenovi\'{c},
Simulation functions: a survey of recent results, Rev. R. Acad. Cienc.
Exactas F\'{\i}s. Nat. Ser. A Mat. RACSAM 113 (2019), no. 3, 2923-2957.

\bibitem{Ercinar} G. Z. Er\c{c}\i nar, Some geometric properties of fixed
points, Ph.D. Thesis, Eski\c{s}ehir Osmangazi University, 2020.

\bibitem{Hieu 2015} N. T. Hieu, N. T. Thanh Ly and N. V. Dung, A
generalization of \'{C}iri\'{c} quasi-contractions for maps on $S$-metric
spaces. Thai J. Math. 13 (2015), no. 2, 369-380.

\bibitem{Joshi} M. Joshi, A. Tomar and S. K. Padaliya, Fixed Point to Fixed
Ellipse in Metric Spaces and Discontinuous Activation Function, To appear in
Applied Mathematics E-Notes.

\bibitem{Karapinar 2017} E. Karap\i nar and F. Khojasteh, An approach to
best proximity points results via simulation functions, J. Fixed Point
Theory Appl. 19 (2017), no. 3, 1983-1995.

\bibitem{Khojasteh} F. Khojasteh, S. Shukla and S. Radenovi\'{c}, A new
approach to the study of fixed point theory for simulation functions,
Filomat 29 (6) (2015), 1189-1194.

\bibitem{Kostic 2019} A. Kosti\'{c}, V. Rako\v{c}evi\'{c} and S. Radenovi%
\'{c}, Best proximity points involving simulation functions with $w_{0}$%
-distance, Rev. R. Acad. Cienc. Exactas F\'{\i}s. Nat. Ser. A Mat. RACSAM
113 (2019), no. 2, 715-727.

\bibitem{Li} L. K. Li, Fixed point analysis for discrete-time recurrent
neural networks, In: Proceedings 1992 IJCNN International Joint Conference
on Neural Networks, Baltimore, MD, USA, 1992, pp. 134-139 vol.4, doi:
10.1109/IJCNN.1992.227277.

\bibitem{Mlaiki 2018} N. Mlaiki, U. \c{C}elik, N. Ta\c{s}, N. \"{O}zg\"{u}r
and A. Mukheimer, Wardowski type contractions and the fixed-circle problem
on $S$-metric spaces, J. Math. 2018, Art. ID 9127486, 9 pp.

\bibitem{Mlaiki 2019} N. Mlaiki, N. Y. \"{O}zg\"{u}r and N. Ta\c{s}, New
fixed-point theorems on an $S$-metric space via simulation functions,
Mathematics, 7 (2019), no. 7, 583.

\bibitem{Mlaiki arxiv} N. Mlaiki, N. \"{O}zg\"{u}r and N. Ta\c{s}, New
fixed-circle results related to $F_{c}$-contractive and $F_{c}$-expanding
mappings on metric spaces, arXiv:2101.10770.

\bibitem{Olgun 2016} M. Olgun, \"{O}. Bi\c{c}er and T. Aly\i ld\i z, A new
aspect to Picard operators with simulation functions, Turkish J. Math. 40
(2016), no. 4, 832-837.

\bibitem{Ozgur 2017} N. Y. \"{O}zg\"{u}r, N. Ta\c{s}, Some new contractive
mappings on $S$-metric spaces and their relationships with the mapping ($S25$%
), Math. Sci. (Springer) 11 (2017), no. 1, 7-16.

\bibitem{Ozgur-Aip} N. Y. \"{O}zg\"{u}r and N. Ta\c{s}, Some fixed-circle
theorems and discontinuity at fixed circle, AIP Conference Proceedings 1926,
020048 (2018).

\bibitem{Ozgur-Tas-circle-thesis} N. Y. \"{O}zg\"{u}r and N. Ta\c{s},
Fixed-circle problem on $S$-metric spaces with a geometric viewpoint, Facta
Universitatis. Series: Mathematics and Informatics 34 (2019), no. 3, 459-472.

\bibitem{Ozgur-malaysian} N. Y. \"{O}zg\"{u}r and N. Ta\c{s}, Some
fixed-circle theorems on metric spaces, Bull. Malays. Math. Sci. Soc. 42
(2019), no. 4, 1433-1449.

\bibitem{Ozgur-simulation} N. \"{O}zg\"{u}r, Fixed-disc results via
simulation functions, Turkish J. Math. 43 (2019), no. 6, 2794-2805.

\bibitem{Pant Ozgur Tas 2019} R. P. Pant, N. Y. \"{O}zg\"{u}r and N. Ta\c{s}%
, On discontinuity problem at fixed point, Bull. Malays. Math. Sci. Soc. 43
(2020), no. 1, 499-517.

\bibitem{Pant Ozgur Tas accepted} R. P. Pant, N. Y. \"{O}zg\"{u}r and N. Ta%
\c{s}, Discontinuity at fixed points with applications, Bull. Belg. Math.
Soc. Simon Stevin 26 (2019), no. 4, 571-589.

\bibitem{Pant 2020} R. P. Pant, N. \"{O}zg\"{u}r, N. Ta\c{s}, A. Pant and M.
Joshi, New results on discontinuity at fixed point, J. Fixed Point Theory
Appl. 22 (2020), no. 2, 39.

\bibitem{Roldan} A. F. Rold\'{a}n-L\'{o}pez-de-Hierro, E. Karap\i nar, C.
Rold\'{a}n-L\'{o}pez-de-Hierro and J. Mart\'{\i}nez-Moreno, Coincidence
point theorems on metric spaces via simulation functions, J. Comput. Appl.
Math. 275 (2015), 345-355.

\bibitem{Roldan 2015} A. F. Rold\'{a}n L\'{o}pez de Hierro and N. Shahzad,
New fixed point theorem under $R$-contractions, Fixed Point Theory Appl.
2015, 2015:98, 18 pp.

\bibitem{Tas-Ozgur-Mlaiki-2} N. Ta\c{s}, N. Y. \"{O}zg\"{u}r and N. Mlaiki,
New types of $F_{c}$-contractions and the fixed-circle problem, Mathematics
6 (2018), 188.

\bibitem{Ramirez} J. L. Ramirez,  Inversions in an ellipse, Forum Geom. 14 (2014), 107-115.

\bibitem{Sedghi 2012} S. Sedghi, N. Shobe, A. Aliouche, A generalization of
fixed point theorems in $S$-metric spaces, Mat. Vesnik 64 (2012), no. 3, 258-266.

\bibitem{Sedghi 2014} S. Sedghi, N. V. Dung, Fixed point theorems on $S$%
-metric spaces, Mat. Vesnik, 66 (2014), no. 1, 113-124.

\bibitem{Tas 2020} N. Ta\c{s}, Bilateral-type solutions to the fixed-circle
problem with rectified linear units application, Turkish J. Math. 44 (2020), no. 4,  1330-1344.
\end{thebibliography}
\end{document}